\documentclass[12pt,fleqn,draft]{article} 
\usepackage{amsmath,amssymb,amsthm,amsfonts,bm}
\usepackage{enumerate}
\usepackage{indentfirst}
\usepackage{cite}
\usepackage[usenames,dvipsnames]{color}

\makeatletter
\def\@cite#1#2{[{{\bfseries #1}\if@tempswa , #2\fi}]}
\renewcommand{\section}{%
\@startsection{section}{1}{\z@}
{0.5truecm plus -1ex minus -.2ex}%
{1.0ex plus .2ex}{\bfseries\large}}
\def\@seccntformat#1{\csname the#1\endcsname.\ }
\makeatother

\numberwithin{equation}{section} 
\theoremstyle{theorem}
\newtheorem{thm}{Theorem}[section]

\newtheorem{lem}[thm]{Lemma}

\theoremstyle{definition}
\newtheorem{df}{Definition}[section]
\newtheorem{remark}{Remark}[section]

\newtheorem*{prth1.1}{Proof of Theorem 1.1}
\newtheorem*{prth1.2}{Proof of Theorem 1.2}
\newtheorem*{prth1.3}{Proof of Theorem 1.3}
\newtheorem*{prth1.4}{Proof of Theorem 1.4}

\let\hat\widehat

\def\Pi{\hat\pi}

\begin{document}
\begin{center}
    \Large{{\bf Time discretization of a nonlocal phase-field system \\ 
                                                                               with inertial term}}
\end{center}
\vspace{5pt}
\begin{center}
    Shunsuke Kurima%
    \\
    \vspace{2pt}
    Department of Mathematics, 
    Tokyo University of Science\\
    1-3, Kagurazaka, Shinjuku-ku, Tokyo 162-8601, Japan\\
    {\tt shunsuke.kurima@gmail.com}\\
    \vspace{2pt}
\end{center}
\begin{center}    
    \small \today
\end{center}

\vspace{2pt}
\newenvironment{summary}
{\vspace{.5\baselineskip}\begin{list}{}{%
     \setlength{\baselineskip}{0.85\baselineskip}
     \setlength{\topsep}{0pt}
     \setlength{\leftmargin}{12mm}
     \setlength{\rightmargin}{12mm}
     \setlength{\listparindent}{0mm}
     \setlength{\itemindent}{\listparindent}
     \setlength{\parsep}{0pt}
     \item\relax}}{\end{list}\vspace{.5\baselineskip}}
\begin{summary}
{\footnotesize {\bf Abstract.} 
Time discretizations of phase-field systems have been studied. 
For example, a time discretization and an error estimate for a parabolic-parabolic phase-field system have been studied by Colli--K. [Commun. Pure Appl. Anal. 18 (2019)].  
Also, a time discretization and an error estimate for a simultaneous abstract evolution equation applying parabolic-hyperbolic phase field systems 
and the linearized equations of coupled sound and heat flow 
have been studied 
(see K. [ESAIM Math. Model. Numer. Anal.54 (2020), 
                                    Electron. J. Differential Equations 2020, Paper No. 96]). 
On the other hand, although existence, continuous dependence estimates  
and behavior of solutions to 
nonlocal phase-field systems with inertial terms have been studied 
by Grasselli--Petzeltov\'a--Schimperna [Quart. Appl. Math. 65 (2007)], 
time discretizations of these systems seem to be not studied yet. 
In this paper we focus on employing a time discretization scheme 
for a nonlocal phase-field system with inertial term  
and establishing an error estimate 
for the difference between continuous and discrete solutions. 
}
\end{summary}
\vspace{10pt}

\newpage

\section{Introduction}\label{Sec1}

Time discretizations of phase-field systems have been studied. 
For example, for the classical phase-field model proposed by Caginalp 
(cf.\ \cite{Cag, EllZheng}; one may also see 
the monographs \cite{BrokSpr, fremond, V1996})  
\begin{equation*}\tag{E1}\label{E1}
     \begin{cases}
         \theta_{t} + \varphi_{t} - \Delta\theta = f   
         & \mbox{in}\ \Omega\times(0, T), 
 \\[0mm]
         \varphi_{t} - \Delta\varphi + \beta(\varphi) + \pi(\varphi) 
         = \theta 
         & \mbox{in}\ \Omega\times(0, T),  
     \end{cases}
 \end{equation*}
Colli--K.\ \cite{CK1} have studied 
a time discretization and an error estimate,  
where $\Omega$ is a domain in $\mathbb{R}^{d}$ ($d \in \mathbb{N}$), 
$T>0$, 
$\beta : \mathbb{R} \to \mathbb{R}$ is a maximal monotone function, 
$\pi : \mathbb{R} \to \mathbb{R}$ is an anti-monotone function, 
$f : \Omega\times(0, T) \to \mathbb{R}$ is a given function. 
Also, for a simultaneous abstract evolution equation applying 
the parabolic-hyperbolic phase-field system 
(see e.g., \cite{GP2003, GP2004, GPS2006, WGZ2007dynamicalBD, WGZ2007})   
\begin{equation*}\tag{E2}\label{E2}
     \begin{cases}
         \theta_{t} + \varphi_{t} - \Delta\theta = f      
         & \mbox{in}\ \Omega\times(0, T), 
 \\[0mm]
        \varphi_{tt} + \varphi_{t} - \Delta\varphi 
        + \beta(\varphi) + \pi(\varphi) = \theta 
         & \mbox{in}\ \Omega\times(0, T),                                
     \end{cases}
\end{equation*}
a time discretization scheme has been employed 
and an error estimate has been derived  
(see \cite{K4}). 
Moreover, for 
a simultaneous abstract evolution equation applying 
\eqref{E2} (in the case that $f=0$) 
and the linearized equations of coupled sound and heat flow 
(see e.g, Matsubara--Yokota \cite{MY2016})
\begin{equation*}\tag{E3}\label{E3}
     \begin{cases}
         \theta_{t} + (\gamma-1)\varphi_{t} - \sigma\Delta\theta = 0      
         & \mbox{in}\ \Omega\times(0, T), 
 \\[0mm]
        \varphi_{tt}  - c^2 \Delta\varphi 
        -m^2 \varphi = -c^2 \Delta\theta 
         & \mbox{in}\ \Omega\times(0, T),                             
     \end{cases}
\end{equation*}
a time discretization and an error estimate 
have been studied, 
where $c>0$, $\sigma>0$, $m \in \mathbb{R}$ and $\gamma>1$ are constants 
(see \cite{K5}).    
On the other hand, Grasselli--Petzeltov\'a--Schimperna \cite{GPS2007}  
have established existence, a continuous dependence estimate 
and behavior of solutions to 
the nonlocal phase-field system  
\begin{equation*}\tag{E4}\label{E4}
     \begin{cases}
         \theta_t + \varphi_t - \Delta \theta = f  
&\mbox{in}\ \Omega \times (0, T), 
\\[0mm] 
\varphi_{tt} + \varphi_t +a(\cdot)\varphi - J\ast\varphi 
                                           + \beta(\varphi) + \pi(\varphi) = \theta
&\mbox{in}\ \Omega \times (0, T),                           
     \end{cases}
\end{equation*}
where $\displaystyle a(x) := \int_{\Omega} J(x-y)\,dy$ for $x \in \Omega$, 
$\displaystyle (J\ast\varphi)(x) := \int_{\Omega} J(x-y)\varphi(y)\,dy$ 
for $x \in \Omega$, 
$J : \mathbb{R}^d \to \mathbb{R}$ is a given function. 
However, time discretizations of \eqref{E4} seem to be not studied yet.   
\medskip

In this paper, 
for the nonlocal phase-field system with inertial term
\begin{equation}\label{P}\tag{P}
\begin{cases}
\theta_t + \varphi_t - \Delta \theta = f  
&\mbox{in}\ \Omega \times (0, T), 
\\[1mm] 
\varphi_{tt} + \varphi_t +a(\cdot)\varphi - J\ast\varphi 
                                           + \beta(\varphi) + \pi(\varphi) = \theta
&\mbox{in}\ \Omega \times (0, T), 
\\[2mm]
\partial_\nu \theta = 0 
&\mbox{on}\ \partial\Omega \times (0, T), 
\\[1mm] 
\theta(0) = \theta_0,\ \varphi(0) = \varphi_0,\ \varphi_{t}(0) = v_0 &\mbox{in}\ \Omega,   
\end{cases}
\end{equation}
we employ the following time discretization scheme:  
find $(\theta_{n+1}, \varphi_{n+1})$  such that  
\begin{equation*}\tag*{(P)$_{n}$}\label{Pn}
     \begin{cases}
        \frac{\theta_{n+1}-\theta_{n}}{h} + \frac{\varphi_{n+1}-\varphi_{n}}{h} 
        - \Delta\theta_{n+1} = f_{n+1}  
         & \mbox{in}\ \Omega, 
 \\[1mm]
         z_{n+1} + v_{n+1} + a(\cdot)\varphi_{n} - J\ast\varphi_{n} 
         + \beta(\varphi_{n+1}) + \pi(\varphi_{n+1}) 
         = \theta_{n} 
         & \mbox{in}\ \Omega, 
 \\[1mm]
         z_{n+1} = \frac{v_{n+1}-v_{n}}{h},\ v_{n+1} = \frac{\varphi_{n+1}-\varphi_{n}}{h} 
         & \mbox{in}\ \Omega, 
 \\[1mm]
         \partial_{\nu}\theta_{n+1} = 0                                   
         & \mbox{on}\ \partial\Omega 
     \end{cases}
\end{equation*}
for $n=0, ... , N-1$, where $h=\frac{T}{N}$, $N \in \mathbb{N}$ 
and $\displaystyle f_{k} := \frac{1}{h}\int_{(k-1)h}^{kh} f(s)\,ds$  
for $k = 1, ... , N$. 
Here $\Omega \subset \mathbb{R}^d$ ($d= 1, 2, 3$) is a bounded domain 
with smooth boundary $\partial\Omega$, 
$\partial_\nu$ denotes differentiation with respect to
the outward normal of $\partial\Omega$, 
$\theta_{0} : \Omega \to \mathbb{R}$, 
$\varphi_{0} : \Omega \to \mathbb{R}$ and 
$v_{0} : \Omega \to \mathbb{R}$   
are given functions. 
Moreover, in this paper 
we assume that 
\begin{enumerate} 
\setlength{\itemsep}{0mm}
\item[(A1)] 
$J(-x) = J(x)$ for all $x \in \mathbb{R}^d$ 
and $\displaystyle\sup_{x \in \Omega} \int_{\Omega} |J(x-y)|\,dy < + \infty$. 
\item[(A2)] $\beta : \mathbb{R} \to \mathbb{R}$                                
is a single-valued maximal monotone function 
such that 
there exists a proper lower semicontinuous convex function 
$\hat{\beta} : \mathbb{R} \to [0, +\infty)$ 
satisfying that   
$\hat{\beta}(0) = 0$ and 
$\beta = \partial\hat{\beta}$, 
where $\partial\hat{\beta}$  
is the subdifferential of $\hat{\beta}$. 
Moreover, $\beta : \mathbb{R} \to \mathbb{R}$ is local Lipschitz continuous. 
\item[(A3)] $\pi : \mathbb{R} \to \mathbb{R}$ is a Lipschitz continuous function. 
\item[(A4)] $f \in L^2(\Omega\times(0, T))$, 
$\theta_0 \in H^1(\Omega) \cap L^{\infty}(\Omega)$, 
$\varphi_0, v_0 \in L^{\infty}(\Omega)$. 
\end{enumerate}
In the case that $\beta(r) = ar^3$, $\widehat{\beta}(r) = \frac{a}{4}r^4$, 
$\pi(r) = br + c$ for $r \in \mathbb{R}$, 
where $a>0$, $b, c \in \mathbb{R}$ are some constants, 
the conditions (A2) and (A3) hold.  

\begin{remark}\label{remark about betahatverphi0}
We see from (A2), (A4) and the definition of the subdifferential 
that 
\[
0 \leq \widehat{\beta}(\varphi_{0}) \leq \beta(\varphi_{0})\varphi_{0} 
\in L^{\infty}(\Omega). 
\]   
\end{remark}

\bigskip

%
%
%
Let us define the Hilbert spaces 
   \[
   H:=L^2(\Omega), \quad V:=H^1(\Omega)
   \]
 with inner products 
 \begin{align*} 
 &(u_{1}, u_{2})_{H}:=\int_{\Omega}u_{1}u_{2}\,dx \quad  (u_{1}, u_{2} \in H), \\
 &(v_{1}, v_{2})_{V}:=
 \int_{\Omega}\nabla v_{1}\cdot\nabla v_{2}\,dx + \int_{\Omega} v_{1}v_{2}\,dx \quad 
 (v_{1}, v_{2} \in V),
\end{align*}
 respectively,
 and with the related Hilbertian norms. 
 Moreover, we use the notation
   \[
   W:=\bigl\{z\in H^2(\Omega)\ |\ \partial_{\nu}z = 0 \quad 
   \mbox{a.e.\ on}\ \partial\Omega\bigr\}.
   \] 

\bigskip

We define solutions of \eqref{P} as follows.
%
%
%
 \begin{df}         
 A pair $(\theta, \varphi)$ with 
    \begin{align*}
    &\theta \in H^1(0, T; H) \cap L^{\infty}(0, T; V) \cap L^{2}(0, T; W),  
    \\
    &\varphi \in W^{2, \infty}(0, T; H) \cap W^{2, 2}(0, T; L^{\infty}(\Omega)) 
                                                            \cap W^{1, \infty}(0, T; L^{\infty}(\Omega))  
    \end{align*}
 is called a {\it solution} of \eqref{P} 
 if $(\theta, \varphi)$ satisfies 
    \begin{align*}
        &\theta_t + \varphi_t - \Delta \theta = f  
                           \quad \mbox{a.e.\ on}\ \Omega\times(0, T), 
     \\[2mm]
        &\varphi_{tt} + \varphi_t +a(\cdot)\varphi - J\ast\varphi 
                                                             + \beta(\varphi) + \pi(\varphi) = \theta 
                               \quad \mbox{a.e.\ on}\ \Omega\times(0, T), 
     \\[2mm]
        &\theta(0) = \theta_0,\ \varphi(0) = \varphi_0,\ \varphi_{t}(0) = v_0 
                                                       \quad \mbox{a.e.\ on}\ \Omega. 
     \end{align*}
 \end{df}

The first main result asserts existence and uniqueness of solutions to 
\ref{Pn} for $n = 0, ..., N-1$. 
\begin{thm}\label{maintheorem1}
Assume that {\rm (A1)-(A4)} hold. 
Then there exists $h_{0} \in (0, 1)$ such that  
for all $h \in (0, h_{0})$ 
there exists a unique solution $(\theta_{n+1}, \varphi_{n+1})$ 
of {\rm \ref{Pn}} 
satisfying 
\[
\theta_{n+1} \in W,\ \varphi_{n+1} \in L^{\infty}(\Omega) 
\quad\mbox{for}\ n=0, ..., N-1.
\] 
\end{thm}
\bigskip

Here, setting  
\begin{align}
&\hat{\theta}_{h}(t) := \theta_{n} + \frac{\theta_{n+1}-\theta_{n}}{h}(t-nh), 
\label{hat1} 
\\[2mm]  
&\hat{\varphi}_{h}(t) := \varphi_{n} + \frac{\varphi_{n+1}-\varphi_{n}}{h}(t-nh), 
\label{hat2} 
\\[2mm]  
&\hat{v}_{h}(t) := v_{n} + \frac{v_{n+1}-v_{n}}{h}(t-nh)  
\label{hat3} 
\end{align}
for $t \in [nh, (n+1)h]$, $n = 0, ..., N-1$, 
and 
\begin{align}
&\overline{\theta}_{h} (t) := \theta_{n+1},\ 
\underline{\theta}_{h} (t) := \theta_{n},\ 
\overline{\varphi}_{h} (t) := \varphi_{n+1},\ 
\underline{\varphi}_{h} (t) := \varphi_{n},\ \label{line1}   
\\[2mm]
&\overline{v}_{h} (t) := v_{n+1},\ 
\overline{z}_{h} (t) := z_{n+1},\ 
\overline{f}_{h}(t) := f_{n+1}  
\label{line2}   
\end{align}
for \ $t \in (nh, (n+1)h]$, $n=0, ..., N-1$, 
we can rewrite \ref{Pn} as  
\begin{equation*}\tag*{(P)$_{h}$}\label{Ph}
     \begin{cases}
        (\hat{\theta}_{h})_{t} + (\hat{\varphi}_{h})_{t} - \Delta\overline{\theta}_{h}  
        = \overline{f}_{h}  
         & \mbox{in}\ \Omega\times(0, T), 
 \\[2mm]
         \overline{z}_{h} + \overline{v}_{h} 
         + a(\cdot)\underline{\varphi}_{h} - J\ast\underline{\varphi}_{h} 
         + \beta(\overline{\varphi}_{h}) + \pi(\overline{\varphi}_{h}) 
         = \underline{\theta}_{h} 
         & \mbox{in}\ \Omega\times(0, T), 
\\[2mm]
         \overline{z}_{h} = (\hat{v}_{h})_{t},\ \overline{v}_{h} = (\hat{\varphi}_{h})_{t} 
         & \mbox{in}\ \Omega\times(0, T), 
 \\[3mm]
         \partial_{\nu}\overline{\theta}_{h} = 0                                   
         & \mbox{on}\ \partial\Omega\times(0, T),
 \\[2mm]
        \hat{\theta}_{h}(0) = \theta_{0},\ \hat{\varphi}_{h}(0) = \varphi_{0},\ 
        \hat{v}_{h}(0) = v_{0}                                      
         & \mbox{in}\ \Omega.  
     \end{cases}
 \end{equation*}
We can prove the following theorem by passing to the limit in \ref{Ph} 
as $h \searrow 0$ (see Section \ref{Sec4}). 
%
%
\begin{thm}\label{maintheorem2}
Assume that {\rm (A1)-(A4)} hold. 
Then there exists a unique solution $(\theta, \varphi)$ of {\rm \eqref{P}}. 
\end{thm}
The following theorem is concerned with the error estimate 
between the solution of \eqref{P} and the solution of \ref{Ph}. 
\begin{thm}\label{maintheorem3} 
Let $h_{0}$ be as in Theorem \ref{maintheorem1}. 
Assume that {\rm (A1)-(A4)} hold. 
Assume further that 
$f \in W^{1,1}(0, T; H)$. 
Then there exist constants $M>0$ and $h_{00} \in (0, h_{0})$ 
depending on the data such that 
\begin{align*} 
&\|\widehat{v}_{h} - v\|_{C([0, T]; H)} 
+ \|\overline{v}_{h} - v\|_{L^{2}(0, T; H)} 
+ \|\widehat{\varphi}_{h} - \varphi\|_{C([0, T]; H)} 
+ \|\widehat{\theta}_{h} - \theta \|_{C([0, T]; H)}
\\ 
&+ \|\nabla(\overline{\theta}_{h} - \theta)\|_{L^{2}(0, T; H)} 
\leq M h^{1/2}    
\end{align*}
for all $h \in (0, h_{00})$, where $v = \varphi_{t}$.  
\end{thm}
\medskip

\begin{remark}
From \eqref{hat1}-\eqref{line2} we can obtain directly the following properties: 
\begin{align}
&\|\widehat{\theta}_{h}\|_{L^{\infty}(0, T; V)} 
= \max\{\|\theta_{0}\|_{V},\ \|\overline{\theta}_{h}\|_{L^{\infty}(0, T; V)}\}, 
\label{rem1} 
\\[3mm] 
&\|\widehat{\varphi}_{h}\|_{L^{\infty}(0, T; L^{\infty}(\Omega))} 
= \max\{\|\varphi_{0}\|_{L^{\infty}(\Omega)},\ 
                      \|\overline{\varphi}_{h}\|_{L^{\infty}(0, T; L^{\infty}(\Omega))}\}, 
\label{rem2} 
\\[3mm] 
&\|\widehat{v}_{h}\|_{L^{\infty}(0, T; L^{\infty}(\Omega))} 
= \max\{\|v_{0}\|_{L^{\infty}(\Omega)},\ 
                      \|\overline{v}_{h}\|_{L^{\infty}(0, T; L^{\infty}(\Omega))}\}, 
\label{rem3} 
\\[3mm] 
&\|\overline{\theta}_{h} - \widehat{\theta}_{h}\|_{L^2(0, T; H)}^2 
   = \dfrac{h^2}{3}\|(\widehat{\theta}_{h})_{t}\|_{L^2(0, T; H)}^2, 
\label{rem4}
\\[3mm]
&\|\overline{\varphi}_{h} - \widehat{\varphi}_{h}\|_{L^{\infty}(0, T; L^{\infty}(\Omega))} 
   = h\|(\widehat{\varphi}_{h})_{t}\|_{L^{\infty}(0, T; L^{\infty}(\Omega))} 
   = h\|\overline{v}_{h}\|_{L^{\infty}(0, T; L^{\infty}(\Omega))},  
\label{rem5}
\\[3mm]
&\|\overline{v}_{h} - \widehat{v}_{h}\|_{L^{\infty}(0, T; H)} 
   = h\|(\widehat{v}_{h})_{t}\|_{L^{\infty}(0, T; H)} 
   = h\|\overline{z}_{h}\|_{L^{\infty}(0, T; H)},  
\label{rem6}
\\[3mm]
&h(\widehat{\theta}_{h})_{t} 
  = \overline{\theta}_{h} - \underline{\theta}_{h}, \label{rem7} 
\\[3mm]
&h(\widehat{\varphi}_{h})_{t} 
  = \overline{\varphi}_{h} - \underline{\varphi}_{h}. \label{rem8} 
\end{align}
\end{remark}
\begin{remark}
Unlike in the case of local parabolic-hyperbolic phase-field systems, 
we cannot 
establish the $L^{p}(0, T; V)$-estimate ($1 \leq p \leq \infty$) 
for $\{\widehat{\varphi}_{h}\}_{h}$ 
and cannot apply the Aubin--Lions lemma 
(see e.g., \cite[Section 8, Corollary 4]{Si-1987}) 
for $\{\widehat{\varphi}_{h}\}_{h}$.  
Thus, since $\pi : \mathbb{R} \to \mathbb{R}$ 
is not monotone, 
to obtain the strong convergence of $\{\widehat{\varphi}_{h}\}_{h}$
in $L^{\infty}(0, T; H)$, 
which is necessary to verify that $\pi(\overline{\varphi}_{h}) \to \pi(\varphi)$ 
strongly in $L^{\infty}(0, T; H)$ as $h=h_{j} \searrow 0$ 
by the Lipschitz continuity of $\pi$ and the property \eqref{rem5}, 
we will try to confirm Cauchy's criterion for solutions of \ref{Ph} 
(see Lemma \ref{Cauchy's criterion}). 
\end{remark}

This paper is organized as follows. 
In Section \ref{Sec2} we prove existence and uniqueness of solutions to \ref{Pn} 
for $n=0, ..., N-1$. 
In Section \ref{Sec3} we derive a priori estimates  
and Cauchy's criterion for solutions of \ref{Ph}. 
Section \ref{Sec4} is devoted to the proofs of 
existence and uniqueness of solutions to \eqref{P} and an error estimate 
between the solution of \eqref{P} and the solution of \ref{Ph}.

\vspace{10pt}

\section{Existence and uniqueness for the discrete problem}\label{Sec2}
In this section we will show Theorem \ref{maintheorem1}.   
\begin{lem}\label{elliptic1}
There exists 
$h_{1} \in (0, \min\{1, 1/\|\pi'\|_{L^{\infty}(\mathbb{R})}\})$
such that 
for all $g \in L^{\infty}(\Omega)$ and all $h \in (0, h_{1})$
there exists a unique solution $\varphi \in L^{\infty}(\Omega)$ of the equation 
\[
\varphi + h\varphi + h^{2} \beta(\varphi) + h^{2} \pi(\varphi) = g 
\quad \mbox{a.e.\ on}\ \Omega.  
\]
\end{lem}
\begin{proof}
We set the operator $\Phi: D(\Phi) \subset H \to H$ as 
\[
\Phi z := h^2 \beta(z) \quad \mbox{for}\ 
z \in D(\Phi) := \{ z \in H\ |\ \beta(z) \in H\}.  
\]
Then this operator is maximal monotone. 
Also, we define the operator $\Psi : H \to H$ as 
\[
\Psi (z) := hz + h^2 \pi(z) \quad \mbox{for}\ z \in H. 
\]
Then this operator is Lipschitz continuous and monotone for all 
$h \in (0, 1/\|\pi'\|_{L^{\infty}(\mathbb{R})})$. 
Thus the operator $\Phi + \Psi : D(\Phi) \subset H \to H$ is maximal monotone 
(see e.g., \cite[Lemma IV.2.1 (p.165)]{S-1997}) and then it follows that 
for all $g \in L^{\infty}(\Omega)$ and all $h \in (0, 1/\|\pi'\|_{L^{\infty}(\mathbb{R})})$ 
there exists a unique solution $\varphi \in D(\Phi)$ of the equation 
\begin{equation}\label{semifinalofproof}
\varphi + h\varphi + h^{2} \beta(\varphi) + h^{2} \pi(\varphi) = g 
\quad \mbox{in}\ H. 
\end{equation}
Here we test \eqref{semifinalofproof} by $\varphi(x)$ 
and use (A3), the Young inequality to infer that  
\begin{align*}
&|\varphi(x)|^2 + h|\varphi(x)|^2 + h^2 \beta(\varphi(x))\varphi(x) 
\\ \notag 
&= g(x)\varphi(x) - h^2 (\pi(\varphi(x))-\pi(0))\varphi(x) - h^2 \pi(0)\varphi(x) 
\\ \notag 
&\leq \frac{1}{2}\|g\|_{L^{\infty}(\Omega)}^2 
         + \frac{1}{2}|\varphi(x)|^2 
         + h^2 \|\pi'\|_{L^{\infty}(\mathbb{R})}|\varphi(x)|^2 
         + \frac{1}{2} h^2 |\varphi(x)|^2 
         + \frac{1}{2} h^2 |\pi(0)|^2 
\end{align*}
for a.a.\ $x \in \Omega$ and all $h \in (0, 1/\|\pi'\|_{L^{\infty}(\mathbb{R})})$.
Hence we see from the monotonicity of $\beta$ that 
there exist constants $C_{1}>0$ and 
$h_{1} \in (0, \min\{1, 1/\|\pi'\|_{L^{\infty}(\mathbb{R})}\})$
such that 
\begin{equation*}
|\varphi(x)|^2 \leq C_{1}
\end{equation*}
for a.a.\ $x \in \Omega$ and all $h \in (0, h_{1})$, which means that 
$\varphi \in L^{\infty}(\Omega)$.  
\end{proof}

\begin{prth1.1}
We can rewrite \ref{Pn} as   
\begin{equation*}\tag*{(Q)$_{n}$}\label{Qn}
     \begin{cases}
         \theta_{n+1} - h\Delta\theta_{n+1} 
         = hf_{n+1} + \varphi_{n} - \varphi_{n+1} + \theta_n, 
     \\[5mm] 
         \varphi_{n+1} + h\varphi_{n+1} + h^2 \beta(\varphi_{n+1}) + h^2 \pi(\varphi_{n+1}) 
         \\[2mm]
         = h^2 \theta_{n} + \varphi_{n} + hv_{n} + h\varphi_{n}  
            -h^2 a(\cdot)\varphi_{n} + h^2 J\ast\varphi_{n}.  
     \end{cases}
 \end{equation*}
It is enough for the proof of Theorem \ref{maintheorem1} 
to establish existence and uniqueness of solutions to \ref{Qn} 
in the case that $n=0$. 
Let $h_{1}$ be as in Lemma \ref{elliptic1}  
and let $h \in (0, h_{1})$. 
Then, owing to (A1), (A4) and Lemma \ref{elliptic1},     
there exists a unique solution $\varphi_{1} \in L^{\infty}(\Omega)$ of  the equation 
\[
\varphi_{1} + h\varphi_{1} + h^2 \beta(\varphi_{1}) + h^2 \pi(\varphi_{1}) 
         = h^2 \theta_{0} + \varphi_{0} + hv_{0} + h\varphi_{0}  
            -h^2 a(\cdot)\varphi_{0} + h^2 J\ast\varphi_{0}.  
\]
Also, for this function $\varphi_1$ there exists a unique solution $\theta_1\in W$ 
of the equation 
\[
\theta_1 - h\Delta\theta_1 
         = hf_1 + \varphi_{0} - \varphi_1 + \theta_0. 
\] 
Therefore 
there exists a unique solution $(\theta_{n+1}, \varphi_{n+1})$ 
of {\rm \ref{Qn}} 
satisfying 
$\theta_{n+1} \in W$ and $\varphi_{n+1} \in L^{\infty}(\Omega)$ 
in the case that $n=0$. 
\qed
\end{prth1.1}

\vspace{10pt}
 
\section{Uniform estimates and Cauchy's criterion}\label{Sec3}

In this section we will derive a priori estimates 
and Cauchy's criterion for solutions of \ref{Ph}.  
\begin{lem}\label{timedisces1}
Let $h_{1}$ be as in Lemma \ref{elliptic1}. 
Then there exist constants $C>0$ and  
$h_{2} \in (0, h_{1})$ 
depending on the data 
such that 
\begin{align*} 
\|\overline{v}_{h}\|_{L^{\infty}(0, T; H)}^2 
+ \|\overline{\varphi}_{h}\|_{L^{\infty}(0, T; H)}^2 
+ \|(\widehat{\theta}_{h})_{t}\|_{L^{2}(0, T; H)}^2 
+ \|\overline{\theta}_{h}\|_{L^{\infty}(0, T; V)}^2  
\leq C 
\end{align*}
for all $h \in (0, h_{2})$.  
\end{lem}
\begin{proof}
We test the first equation in \ref{Pn} by $\theta_{n+1} - \theta_{n}$, 
integrate over $\Omega$  
and use the identities $a(a-b) = \frac{1}{2}a^2 - \frac{1}{2}b^2 + \frac{1}{2}(a-b)^2$ 
($a, b \in \mathbb{R}$) and $v_{n+1} = \frac{\varphi_{n+1}-\varphi_{n}}{h}$, 
the Young inequality  
to infer that 
\begin{align}\label{a1}
&h\Bigl\|\frac{\theta_{n+1} - \theta_{n}}{h}\Bigr\|_{H}^2 
+ \frac{1}{2}\|\theta_{n+1}\|_{V}^2 - \frac{1}{2}\|\theta_{n}\|_{V}^2 
+ \frac{1}{2}\|\theta_{n+1} - \theta_{n}\|_{V}^2 
\\ \notag 
&= h\Bigl(f_{n+1}, \frac{\theta_{n+1} - \theta_{n}}{h} \Bigr)_{H} 
     -h\Bigl(v_{n+1}, \frac{\theta_{n+1} - \theta_{n}}{h} \Bigr)_{H} 
     + h\Bigl(\theta_{n+1}, \frac{\theta_{n+1} - \theta_{n}}{h} \Bigr)_{H} 
\\ \notag 
&\leq \frac{3}{2}h\|f_{n+1}\|_{H}^2 + \frac{3}{2}h\|v_{n+1}\|_{H}^2 
         + \frac{3}{2}h\|\theta_{n+1}\|_{V}^2 
         + \frac{1}{2}h\Bigl\|\frac{\theta_{n+1} - \theta_{n}}{h}\Bigr\|_{H}^2. 
\end{align}
Multiplying the second equation in \ref{Pn} by $hv_{n+1}$, 
integrating over $\Omega$ 
and 
applying the identity $a(a-b) = \frac{1}{2}a^2 - \frac{1}{2}b^2 + \frac{1}{2}(a-b)^2$ 
($a, b \in \mathbb{R}$), 
we see from (A1), (A3) and the Young inequality that 
there exists a constant $C_{1}>0$ such that 
\begin{align}\label{a2}
&\frac{1}{2}\|v_{n+1}\|_{H}^2 - \frac{1}{2}\|v_{n}\|_{H}^2 
+ \frac{1}{2}\|v_{n+1} - v_{n}\|_{H}^2 
+ h\|v_{n+1}\|_{H}^2 
+ (\beta(\varphi_{n+1}), \varphi_{n+1} - \varphi_{n})_{H} 
\\ \notag 
&= h(\theta_{n}, v_{n+1})_{H}  
    - h(a(\cdot)\varphi_{n} - J\ast\varphi_{n}, v_{n+1})_{H} 
   - h(\pi(\varphi_{n+1}) - \pi(0), v_{n+1})_{H} 
\\ \notag 
&\,\quad - h(\pi(0), v_{n+1})_{H} 
\\ \notag 
&\leq \frac{1}{2}h\|\theta_{n}\|_{H}^2 + C_{1}h\|\varphi_{n}\|_{H}^2 
         + \frac{\|\pi'\|_{L^{\infty}(\mathbb{R})}^2}{2}h\|\varphi_{n+1}\|_{H}^2 
         + 2h\|v_{n+1}\|_{H}^2 + \frac{\|\pi(0)\|_{H}^2}{2}h
\end{align}
for all $h \in (0, h_{1})$. 
Here it follows from (A2) and the definition of the subdifferential that
\begin{align}\label{a3}
(\beta(\varphi_{n+1}), \varphi_{n+1} - \varphi_{n})_{H} 
\geq 
\int_{\Omega} \widehat{\beta}(\varphi_{n+1}) - \int_{\Omega}\widehat{\beta}(\varphi_{n})  
\end{align}
and we have from 
the identities $a(a-b) = \frac{1}{2}a^2 - \frac{1}{2}b^2 + \frac{1}{2}(a-b)^2$ 
($a, b \in \mathbb{R}$) and $h v_{n+1} = \varphi_{n+1} - \varphi_{n}$, 
the Young inequality   
that 
\begin{align}\label{a4}
&\frac{1}{2}\|\varphi_{n+1}\|_{H}^2 - \frac{1}{2}\|\varphi_{n}\|_{H}^2 
+ \frac{1}{2}\|\varphi_{n+1} - \varphi_{n}\|_{H}^2 
\\ \notag 
&=  
(\varphi_{n+1}, \varphi_{n+1} - \varphi_{n})_{H} 
= h(\varphi_{n+1}, v_{n+1})_{H}
\leq \frac{1}{2}h\|\varphi_{n+1}\|_{H}^2 + \frac{1}{2}h\|v_{n+1}\|_{H}^2.  
\end{align}
Hence, owing to \eqref{a1}-\eqref{a4} 
and summing over $n=0, ..., m-1$ with $1 \leq m \leq N$, 
it holds that  
\begin{align*}
&\frac{1}{2}h\sum_{n=0}^{m-1}\Bigl\|\frac{\theta_{n+1} - \theta_{n}}{h}\Bigr\|_{H}^2 
+ \frac{1}{2}\|\theta_{m}\|_{V}^2 
+ \frac{1}{2}h^2 \sum_{n=0}^{m-1}\Bigl\|\frac{\theta_{n+1} - \theta_{n}}{h}\Bigr\|_{V}^2 
+ \frac{1}{2}\|v_{m}\|_{H}^2 
\\ \notag 
&+ \frac{1}{2}h^2 \sum_{n=0}^{m-1}\|z_{n+1}\|_{H}^2 
+ h \sum_{n=0}^{m-1}\|v_{n+1}\|_{H}^2 
+ \int_{\Omega}\widehat{\beta}(\varphi_{m}) 
+ \frac{1}{2}\|\varphi_{m}\|_{H}^2 
  + \frac{1}{2}h^2 \sum_{n=0}^{m-1}\|v_{n+1}\|_{H}^2 
\\ \notag 
&\leq \frac{1}{2}\|\theta_{0}\|_{V}^2 + \frac{1}{2}\|v_{0}\|_{H}^2 
        + \int_{\Omega}\widehat{\beta}(\varphi_{0}) 
        + \frac{1}{2}\|\varphi_{0}\|_{H}^2  
        + \frac{3}{2}h\sum_{n=0}^{m-1}\|f_{n+1}\|_{H}^2 
\\ \notag 
    &\,\quad+ 4h\sum_{n=0}^{m-1}\|v_{n+1}\|_{H}^2 
         + \frac{3}{2}h\sum_{n=0}^{m-1}\|\theta_{n+1}\|_{V}^2 
+ \frac{1}{2}h\sum_{n=0}^{m-1}\|\theta_{n}\|_{H}^2 
\\ \notag 
    &\,\quad + C_{1}h\sum_{n=0}^{m-1}\|\varphi_{n}\|_{H}^2 
  + \frac{\|\pi'\|_{L^{\infty}(\mathbb{R})}^2+1}{2}h\sum_{n=0}^{m-1}\|\varphi_{n+1}\|_{H}^2 
         + \frac{\|\pi(0)\|_{H}^2}{2}T 
\end{align*}
for all $h \in (0, h_{1})$ and $m = 1, ..., N$. 
Then the inequality 
\begin{align*}
&\frac{1}{2}h\sum_{n=0}^{m-1}\Bigl\|\frac{\theta_{n+1} - \theta_{n}}{h}\Bigr\|_{H}^2 
+ \frac{1-3h}{2}\|\theta_{m}\|_{V}^2 
+ \frac{1}{2}h^2 \sum_{n=0}^{m-1}\Bigl\|\frac{\theta_{n+1} - \theta_{n}}{h}\Bigr\|_{V}^2 
\\ \notag 
&+ \frac{1-8h}{2}\|v_{m}\|_{H}^2 
+ \frac{1}{2}h^2 \sum_{n=0}^{m-1}\|z_{n+1}\|_{H}^2 
+ h \sum_{n=0}^{m-1}\|v_{n+1}\|_{H}^2 
+ \int_{\Omega}\widehat{\beta}(\varphi_{m}) 
\\ \notag 
&+ \frac{1 - (\|\pi'\|_{L^{\infty}(\mathbb{R})}^2+1)h}{2}\|\varphi_{m}\|_{H}^2 
  + \frac{1}{2}h^2 \sum_{n=0}^{m-1}\|v_{n+1}\|_{H}^2 
\\ \notag 
&\leq \frac{1}{2}\|\theta_{0}\|_{V}^2 + \frac{1}{2}\|v_{0}\|_{H}^2 
        + \int_{\Omega}\widehat{\beta}(\varphi_{0}) 
        + \frac{1}{2}\|\varphi_{0}\|_{H}^2  
        + \frac{3}{2}h\sum_{n=0}^{m-1}\|f_{n+1}\|_{H}^2 
\\ \notag 
    &\,\quad+ 4h\sum_{j=0}^{m-1}\|v_{j}\|_{H}^2 
                                        + 2h\sum_{j=0}^{m-1}\|\theta_{j}\|_{V}^2 
         + \frac{2C_{1} + \|\pi'\|_{L^{\infty}(\mathbb{R})}^2+1}{2}h
                                                           \sum_{j=0}^{m-1}\|\varphi_{j}\|_{H}^2 
         + \frac{\|\pi(0)\|_{H}^2}{2}T 
\end{align*}
holds for all $h \in (0, h_{1})$ and $m = 1, ..., N$. 
Thus there exist constants $C_{2}>0$ and $h_{2} \in (0, h_{1})$ such that  
\begin{align*}
&h\sum_{n=0}^{m-1}\Bigl\|\frac{\theta_{n+1} - \theta_{n}}{h}\Bigr\|_{H}^2 
+ \|\theta_{m}\|_{V}^2 
+ h^2 \sum_{n=0}^{m-1}\Bigl\|\frac{\theta_{n+1} - \theta_{n}}{h}\Bigr\|_{V}^2 
\\ \notag 
&+ \|v_{m}\|_{H}^2 
+ h^2 \sum_{n=0}^{m-1}\|z_{n+1}\|_{H}^2 
+ h \sum_{n=0}^{m-1}\|v_{n+1}\|_{H}^2 
+ \int_{\Omega}\widehat{\beta}(\varphi_{m}) 
\\ \notag 
&+ \|\varphi_{m}\|_{H}^2 
  + h^2 \sum_{n=0}^{m-1}\|v_{n+1}\|_{H}^2 
\\ \notag 
&\leq C_{2}
         + C_{2}h\sum_{j=0}^{m-1}\|v_{j}\|_{H}^2 
         + C_{2}h\sum_{j=0}^{m-1}\|\theta_{j}\|_{V}^2 
         + C_{2}h\sum_{j=0}^{m-1}\|\varphi_{j}\|_{H}^2    
\end{align*}
for all $h \in (0, h_{2})$ and $m = 1, ..., N$. 
Therefore, thanks to the discrete Gronwall lemma (see e.g., \cite[Prop.\ 2.2.1]{Jerome}), 
we can obtain that there exists a constant $C_{3}>0$ such that 
\begin{align*}
&h\sum_{n=0}^{m-1}\Bigl\|\frac{\theta_{n+1} - \theta_{n}}{h}\Bigr\|_{H}^2 
+ \|\theta_{m}\|_{V}^2 
+ h^2 \sum_{n=0}^{m-1}\Bigl\|\frac{\theta_{n+1} - \theta_{n}}{h}\Bigr\|_{V}^2 
\\ \notag 
&+ \|v_{m}\|_{H}^2 
+ h^2 \sum_{n=0}^{m-1}\|z_{n+1}\|_{H}^2 
+ h \sum_{n=0}^{m-1}\|v_{n+1}\|_{H}^2 
+ \int_{\Omega}\widehat{\beta}(\varphi_{m}) 
\\ \notag 
&+ \|\varphi_{m}\|_{H}^2 
  + h^2 \sum_{n=0}^{m-1}\|v_{n+1}\|_{H}^2 
\leq C_{3}
\end{align*}
for all $h \in (0, h_{2})$ and $m = 1, ..., N$. 
\end{proof}
\begin{lem}\label{timedisces2}
Let $h_{2}$ be as in Lemma \ref{timedisces1}. 
Then there exists a constant $C>0$ depending on the data such that 
\begin{align*} 
\|\overline{\theta}_{h}\|_{L^{2}(0, T; W)}  \leq C
\end{align*}
for all $h \in (0, h_{2})$.  
\end{lem}
\begin{proof}
We have from the first equation in \ref{Ph} and Lemma \ref{timedisces1} that 
there exists a constant $C_{1}>0$ such that 
\begin{equation}\label{b1}
\|-\Delta \overline{\theta}_{h}\|_{L^{2}(0, T; H)} \leq C_{1} 
\end{equation}
for all $h \in (0, h_{2})$. 
Thus we can prove Lemma \ref{timedisces2} 
by Lemma \ref{timedisces1}, \eqref{b1} and the elliptic regularity theory. 
\end{proof}
\begin{lem}\label{timedisces3}
Let $h_{2}$ be as in Lemma \ref{timedisces1}. 
Then there exist constants $C>0$ and $h_{3} \in (0, h_{2})$ 
depending on the data such that 
\begin{align*} 
\|\overline{v}_{h}\|_{L^{\infty}(\Omega\times(0, T))}^2 
+ \|\overline{\varphi}_{h}\|_{L^{\infty}(\Omega\times(0, T))}^2 
\leq C
\end{align*}
for all $h \in (0, h_{3})$.  
\end{lem}
\begin{proof}
We derive from the identities $a(a-b) = \frac{1}{2}a^2 - \frac{1}{2}b^2 + \frac{1}{2}(a-b)^2$ 
($a, b \in \mathbb{R}$) and $h v_{n+1} = \varphi_{n+1} - \varphi_{n}$, 
the Young inequality  
that 
\begin{align}\label{c1}
&\frac{1}{2}|\varphi_{n+1}(x)|^2 - \frac{1}{2}|\varphi_{n}(x)|^2 
+ \frac{1}{2}|\varphi_{n+1}(x) - \varphi_{n}(x)|^2 
\\ \notag 
&= \varphi_{n+1}(x)(\varphi_{n+1}(x) - \varphi_{n}(x))
\\ \notag 
&= h\varphi_{n+1}(x)v_{n+1}(x) 
\\ \notag 
&\leq \frac{1}{2}h\|\varphi_{n+1}\|_{L^{\infty}(\Omega)}^2 
         + \frac{1}{2}h\|v_{n+1}\|_{L^{\infty}(\Omega)}^2.  
\end{align}
Testing the second equation in \ref{Ph} by $h v_{n+1}(x)$ 
and using (A1), (A3), the Young inequality 
mean that there exists a constant $C_{1}>0$ such that 
\begin{align}\label{c2}
&\frac{1}{2}|v_{n+1}(x)|^2 - \frac{1}{2}|v_{n}(x)|^2 
+ \frac{1}{2}|v_{n+1}(x) - v_{n}(x)|^2 
\\ \notag 
&\,\quad+ \beta(\varphi_{n+1}(x))(\varphi_{n+1}(x) - \varphi_{n}(x)) 
\\ \notag 
& = h\bigl(\theta_{n}(x) -a(x)\varphi_{n}(x) + (J\ast\varphi_{n})(x) 
                  + \pi(0) - \pi(\varphi_{n+1}(x)) - \pi(0) \bigr)v_{n+1}(x) 
\\ \notag 
&\leq \frac{1}{2}h\|\theta_{n}\|_{L^{\infty}(\Omega)}^2 
         + \frac{1}{2}h\|-a(\cdot)\varphi_{n} + J\ast\varphi_{n}\|_{L^{\infty}(\Omega)}^2 
\\ \notag 
   &\,\quad+ \frac{\|\pi'\|_{L^{\infty}(\mathbb{R})}^2 }{2}h
                    \|\varphi_{n+1}\|_{L^{\infty}(\Omega)}^2                                                                                                 
      + \frac{|\pi(0)|^2}{2}h  
      + 2h\|v_{n+1}\|_{L^{\infty}(\Omega)}^2 
\\ \notag 
&\leq \frac{1}{2}h\|\theta_{n}\|_{L^{\infty}(\Omega)}^2 
         + C_{1}h\|\varphi_{n}\|_{L^{\infty}(\Omega)}^2 
\\ \notag 
   &\,\quad+ \frac{\|\pi'\|_{L^{\infty}(\mathbb{R})}^2 }{2}h
                    \|\varphi_{n+1}\|_{L^{\infty}(\Omega)}^2                                                                                                 
      + \frac{|\pi(0)|^2}{2}h  
      + 2h\|v_{n+1}\|_{L^{\infty}(\Omega)}^2 
\end{align}
for all $h \in (0, h_{2})$ and a.a.\ $x \in \Omega$. 
Here the condition (A2) and the definition of the subdifferential lead to the inequality  
\begin{align}\label{c3}
\beta(\varphi_{n+1}(x))(\varphi_{n+1}(x) - \varphi_{n}(x)) 
\geq \widehat{\beta}(\varphi_{n+1}(x)) - \widehat{\beta}(\varphi_{n}(x)). 
\end{align}
Thus it follows from \eqref{c1}-\eqref{c3}, 
summing over $n=0, ..., m-1$ with $1 \leq m \leq N$ 
and Remark \ref{remark about betahatverphi0} 
that   
\begin{align*}
&\frac{1}{2}|\varphi_{m}(x)|^2 + \frac{1}{2}|v_{m}(x)|^2 + \widehat{\beta}(\varphi_{m}(x)) 
\\ \notag 
&\leq \frac{1}{2}\|\varphi_{0}\|_{L^{\infty}(\Omega)}^2 
         + \frac{1}{2}\|v_{0}\|_{L^{\infty}(\Omega)}^2 
         + \|\widehat{\beta}(\varphi_{0})\|_{L^{\infty}(\Omega)} 
\\ \notag 
    &\,\quad+ \frac{1}{2}h\sum_{n=0}^{m-1}\|\theta_{n}\|_{L^{\infty}(\Omega)}^2 
         + C_{1}h\sum_{n=0}^{m-1}\|\varphi_{n}\|_{L^{\infty}(\Omega)}^2 
\\ \notag 
   &\,\quad + \frac{\|\pi'\|_{L^{\infty}(\mathbb{R})}^2 + 1}{2}h
                                        \sum_{n=0}^{m-1} \|\varphi_{n+1}\|_{L^{\infty}(\Omega)}^2 
+ \frac{5}{2}h\sum_{n=0}^{m-1}\|v_{n+1}\|_{L^{\infty}(\Omega)}^2 
+ \frac{|\pi(0)|^2}{2} T
\end{align*}
for all $h \in (0, h_{2})$, $m = 1, ..., N$ and a.a.\ $x \in \Omega$, which implies that   
\begin{align*}
&\frac{1}{2}\|\varphi_{m}\|_{L^{\infty}(\Omega)}^2 
+ \frac{1}{2}\|v_{m}\|_{L^{\infty}(\Omega)}^2 
\\ \notag 
&\leq \frac{1}{2}\|\varphi_{0}\|_{L^{\infty}(\Omega)}^2 
         + \frac{1}{2}\|v_{0}\|_{L^{\infty}(\Omega)}^2 
         + \|\widehat{\beta}(\varphi_{0})\|_{L^{\infty}(\Omega)} 
\\ \notag 
    &\,\quad+ \frac{1}{2}h\sum_{n=0}^{m-1}\|\theta_{n}\|_{L^{\infty}(\Omega)}^2 
         + C_{1}h\sum_{n=0}^{m-1}\|\varphi_{n}\|_{L^{\infty}(\Omega)}^2 
\\ \notag 
   &\,\quad + \frac{\|\pi'\|_{L^{\infty}(\mathbb{R})}^2 + 1}{2}h
                                        \sum_{n=0}^{m-1} \|\varphi_{n+1}\|_{L^{\infty}(\Omega)}^2 
+ \frac{5}{2}h\sum_{n=0}^{m-1}\|v_{n+1}\|_{L^{\infty}(\Omega)}^2 
+ \frac{|\pi(0)|^2}{2} T
\end{align*}
for all $h \in (0, h_{2})$ and $m = 1, ..., N$. 
Then the inequality 
\begin{align}\label{c4}
&\frac{1- (\|\pi'\|_{L^{\infty}(\mathbb{R})}^2 + 1)h}{2}
                                                     \|\varphi_{m}\|_{L^{\infty}(\Omega)}^2 
+ \frac{1-5h}{2}\|v_{m}\|_{L^{\infty}(\Omega)}^2 
\\ \notag 
&\leq \frac{1}{2}\|\varphi_{0}\|_{L^{\infty}(\Omega)}^2 
         + \frac{1}{2}\|v_{0}\|_{L^{\infty}(\Omega)}^2 
         + \|\widehat{\beta}(\varphi_{0})\|_{L^{\infty}(\Omega)}
\\ \notag 
    &\,\quad+ \frac{1}{2}h\sum_{j=0}^{m-1}\|\theta_{j}\|_{L^{\infty}(\Omega)}^2 
+ \frac{2C_{1} + \|\pi'\|_{L^{\infty}(\mathbb{R})}^2 + 1}{2}h
                                        \sum_{j=0}^{m-1} \|\varphi_{j}\|_{L^{\infty}(\Omega)}^2 
\\ \notag 
   &\,\quad+ \frac{5}{2}h\sum_{j=0}^{m-1}\|v_{j}\|_{L^{\infty}(\Omega)}^2 
+ \frac{|\pi(0)|^2}{2} T
\end{align}
holds for all $h \in (0, h_{2})$ and $m = 1, ..., N$. 
Here we see from 
the continuity of the embedding $W \hookrightarrow L^{\infty}(\Omega)$
and Lemma \ref{timedisces2}  
that there exist constants $C_{2}, C_{3} > 0$ such that 
\begin{align}\label{c5} 
\displaystyle h\sum_{j=1}^{N}\|\theta_{j}\|_{L^{\infty}(\Omega)}^2 = 
\|\overline{\theta}_{h}\|_{L^2(0, T; L^{\infty}(\Omega))}^2 
\leq C_{2}\|\overline{\theta}_{h}\|_{L^2(0, T; W)}^2 
\leq C_{3} 
\end{align}
for all $h \in (0, h_{2})$. 
Therefore we have from \eqref{c4} and \eqref{c5} 
that there exist constants $C_{4}>0$ and $h_{3} \in (0, h_{2})$ such that  
\begin{align*}
&\|\varphi_{m}\|_{L^{\infty}(\Omega)}^2 +\|v_{m}\|_{L^{\infty}(\Omega)}^2 
\\ \notag 
&\leq C_{4} 
           + C_{4}h\sum_{j=0}^{m-1} \|\varphi_{j}\|_{L^{\infty}(\Omega)}^2 
             + C_{4}h\sum_{j=0}^{m-1}\|v_{j}\|_{L^{\infty}(\Omega)}^2 
\end{align*}
for all $h \in (0, h_{3})$ and $m = 1, ..., N$. 
Then we can obtain that 
there exists a constant $C_{5}>0$ such that 
\begin{equation*}
\|\varphi_{m}\|_{L^{\infty}(\Omega)}^2 +\|v_{m}\|_{L^{\infty}(\Omega)}^2 
\leq C_{5} 
\end{equation*}
for all $h \in (0, h_{3})$ and $m = 1, ..., N$ 
by the discrete Gronwall lemma (see e.g., \cite[Prop.\ 2.2.1]{Jerome}). 
\end{proof}
\begin{lem}\label{timedisces4}
Let $h_{3}$ be as in Lemma \ref{timedisces3}. 
Then there exists a constant $C>0$ depending on the data such that 
\begin{align*} 
\|\underline{\varphi}_{h}\|_{L^{\infty}(\Omega\times(0, T))}^2 
+ \|\underline{\theta}_{h}\|_{L^{\infty}(0, T; V) \cap L^{2}(0, T; L^{\infty}(\Omega))}^2 
\leq C
\end{align*}
for all $h \in (0, h_{3})$.  
\end{lem}
\begin{proof}
We can verify this lemma by Lemmas \ref{timedisces1}-\ref{timedisces3}, 
the continuity of the embedding $W \hookrightarrow L^{\infty}(\Omega)$ and (A4). 
\end{proof}
\begin{lem}\label{timedisces5}
Let $h_{3}$ be as in Lemma \ref{timedisces3}. 
Then there exists a constant $C>0$ depending on the data such that 
\begin{align*} 
\|\overline{z}_{h}\|_{L^2(0, T; L^{\infty}(\Omega))}^2 \leq C
\end{align*}
for all $h \in (0, h_{3})$.  
\end{lem}
\begin{proof}
Since it follows from Lemma \ref{timedisces3} and the continuity of $\beta$ 
that there exists a constant $C_{1}>0$ such that 
\begin{equation*}
\|\beta(\overline{\varphi}_{h})\|_{L^{\infty}(\Omega\times(0, T))} \leq C_{1}
\end{equation*}
for all $h \in (0, h_{3})$, 
we can confirm that Lemma \ref{timedisces5} holds 
by the second equation in \ref{Ph}, (A1), (A3), 
Lemmas \ref{timedisces3} and \ref{timedisces4}. 
\end{proof}
\begin{lem}\label{timedisces6}
Let $h_{3}$ be as in Lemma \ref{timedisces3}. 
Then there exist constants $C>0$ and $h_{4} \in (0, h_{3})$ 
depending on the data such that 
\begin{align*} 
\|\overline{z}_{h}\|_{L^{\infty}(0, T; H)}^2 \leq C 
\end{align*}
for all $h \in (0, h_{4})$.  
\end{lem}
\begin{proof}
Since the second equation in \ref{Pn} leads to the identity 
\begin{align*}
z_{1} + hz_{1} + a(\cdot)\varphi_{0} - J\ast\varphi_{0} 
+ \beta(\varphi_{1}) + \pi(\varphi_{1}) = \theta_{0}, 
\end{align*}
it holds that 
\begin{align*}
&\|z_{1}\|_{H}^2 + h\|z_{1}\|_{H}^2 
\\ \notag 
&= - (a(\cdot)\varphi_{0} - J\ast\varphi_{0}, z_{1})_{H} 
  - (\beta(\varphi_{1}), z_{1})_{H} - (\pi(\varphi_{1}) , z_{1})_{H}
  + (\theta_{0}, z_{1})_{H}. 
\end{align*}
Thus we deduce from the Young inequality, (A1), the continuity of $\beta$, (A3), 
and Lemma \ref{timedisces3} that there exists a constant $C_{1}>0$ such that  
\begin{equation}\label{d1} 
\|z_{1}\|_{H}^2 \leq C_{1}
\end{equation}
for all $h \in (0, h_{3})$. 
Now we let $n \in \{1, ..., N-1\}$. 
Then we have from the second equation in \ref{Pn} that 
\begin{align}\label{d2}
&z_{n+1} - z_{n} + hz_{n+1} 
+ ha(\cdot)v_{n} - hJ\ast v_{n} 
+ \beta(\varphi_{n+1}) - \beta(\varphi_{n}) 
\\ \notag 
&+ \pi(\varphi_{n+1}) - \pi(\varphi_{n}) 
= \theta_{n} - \theta_{n-1}. 
\end{align}
Moreover, we test \eqref{d2} by $z_{n+1}$, integrate over $\Omega$, 
recall (A1), Lemma \ref{timedisces3}, 
the local Lipschitz continuity of $\beta$, (A3), 
and use the Young inequality to infer that  
there exist constants $C_{2}, C_{3} > 0$ such that 
\begin{align}\label{d3}
&\frac{1}{2}\|z_{n+1}\|_{H}^2 - \frac{1}{2}\|z_{n}\|_{H}^2 
+ \frac{1}{2}\| z_{n+1} - z_{n}\|_{H}^2  
+ h\|z_{n+1}\|_{H}^2 
\\ \notag 
&= -h(a(\cdot)v_{n} - J\ast v_{n}, z_{n+1})_{H} 
    - h\Bigl(\frac{\beta(\varphi_{n+1}) - \beta(\varphi_{n})}{h}, z_{n+1} \Bigr)_{H} 
\\ \notag 
      &\,\quad- h\Bigl(\frac{\pi(\varphi_{n+1}) - \pi(\varphi_{n})}{h}, z_{n+1} \Bigr)_{H} 
    + h\Bigl(\frac{\theta_{n} - \theta_{n-1}}{h}, z_{n+1} \Bigr)_{H} 
\\ \notag 
&\leq C_{2}h\|v_{n}\|_{H}\|z_{n+1}\|_{H} + C_{2}h\|v_{n+1}\|_{H}\|z_{n+1}\|_{H} 
         + h\Bigl\|\frac{\theta_{n} - \theta_{n-1}}{h}\Bigr\|_{H}\|z_{n+1}\|_{H} 
\\ \notag 
&\leq C_{3}h\|z_{n+1}\|_{H} 
            + h\Bigl\|\frac{\theta_{n} - \theta_{n-1}}{h}\Bigr\|_{H}\|z_{n+1}\|_{H} 
\\ \notag 
&\leq h\|z_{n+1}\|_{H}^2 
        + \frac{1}{2}h\Bigl\|\frac{\theta_{n} - \theta_{n-1}}{h}\Bigr\|_{H}^2 
        + \frac{C_{3}^2}{2}h
\end{align}
for all $h \in (0, h_{3})$. 
Thus, summing \eqref{d3} over $n = 1, ..., \ell-1$ with $2 \leq \ell \leq N$, 
we see from \eqref{d1} and Lemma \ref{timedisces1} that 
there exists a constant $C_{4}>0$ such that 
\begin{align*}
\frac{1}{2}\|z_{\ell}\|_{H}^2 
&\leq \frac{1}{2}\|z_{1}\|_{H}^2 + h\sum_{n=1}^{\ell-1}\|z_{n+1}\|_{H}^2 
       +\frac{1}{2}h\sum_{n=1}^{\ell-1}\Bigl\|\frac{\theta_{n} - \theta_{n-1}}{h}\Bigr\|_{H}^2 
       + \frac{C_{3}^2}{2}T   
\\ \notag 
&\leq C_{4} + h\sum_{n=1}^{\ell-1}\|z_{n+1}\|_{H}^2
\end{align*}
for all $h \in (0, h_{3})$ and $\ell = 2, ..., N$, whence 
we have from \eqref{d1} that 
there exist constants $C_{5}>0$ and $h_{4} \in (0, h_{3})$ such that  
\begin{equation*}
\|z_{m}\|_{H}^2 
\leq C_{5} + C_{5}h\sum_{j=0}^{m-1}\|z_{j}\|_{H}^2
\end{equation*}
for all $h \in (0, h_{4})$ and $m = 1, ..., N$. 
Therefore the discrete Gronwall lemma (see e.g., \cite[Prop.\ 2.2.1]{Jerome}) implies 
that there exists a constant $C_{6}>0$ such that 
\begin{equation*}
\|z_{m}\|_{H}^2 \leq C_{6} 
\end{equation*}
for all $h \in (0, h_{4})$ and $m = 1, ..., N$. 
\end{proof}
\begin{lem}\label{timedisces7}
Let $h_{4}$ be as in Lemma \ref{timedisces6}. 
Then there exists a constant $C>0$ depending on the data such that 
\begin{align*} 
&\|\widehat{\theta}_{h}\|_{H^{1}(0, T; H) \cap L^{\infty}(0, T; V)} 
+ \|\widehat{\varphi}_{h}\|_{W^{1, \infty}(0, T; L^{\infty}(\Omega))} 
\\ 
&+ \|\widehat{v}_{h}\|_{W^{1, \infty}(0, T; H) \cap W^{1, 2}(0, T; L^{\infty}(\Omega)) 
                                                          \cap L^{\infty}(0, T; L^{\infty}(\Omega))} 
\leq C  
\end{align*}
for all $h \in (0, h_{4})$.  
\end{lem}
\begin{proof}
This lemma can be proved by \eqref{rem1}-\eqref{rem3}, 
Lemmas \ref{timedisces1}, \ref{timedisces3}, \ref{timedisces5} and \ref{timedisces6}. 
\end{proof}
\medskip

The following lemma asserts Cauchy's criterion for solutions of \ref{Ph}. 
\begin{lem}\label{Cauchy's criterion}
Let $h_{4}$ be as in Lemma \ref{timedisces6}. 
Then there exists a constant $C>0$  
depending on the data such that 
\begin{align*} 
&\|\widehat{v}_{h} - \widehat{v}_{\tau}\|_{C([0, T]; H)} 
+ \|\overline{v}_{h} - \overline{v}_{\tau}\|_{L^{2}(0, T; H)} 
+ \|\widehat{\varphi}_{h} - \widehat{\varphi}_{\tau}\|_{C([0, T]; H)} 
+ \|\widehat{\theta}_{h} - \widehat{\theta}_{\tau} \|_{C([0, T]; H)}
\\ \notag 
&+ \|\nabla(\overline{\theta}_{h} - \overline{\theta}_{\tau})\|_{L^{2}(0, T; H)}  
\leq C(h^{1/2} + \tau^{1/2}) 
               + C\|\overline{f}_{h} - \overline{f}_{\tau}\|_{L^{2}(0, T; H)}  
\end{align*}
for all $h, \tau \in (0, h_{4})$.  
\end{lem}
\begin{proof}
It holds that 
\begin{align}\label{s1}
&\frac{1}{2}\frac{d}{ds}\|\widehat{v}_{h}(s) - \widehat{v}_{\tau}(s)\|_{H}^2 
= (\overline{z}_{h}(s) - \overline{z}_{\tau}(s), 
                                          \widehat{v}_{h}(s) - \widehat{v}_{\tau}(s))_{H} 
\\ \notag 
&= (\overline{z}_{h}(s) - \overline{z}_{\tau}(s), 
                                          \widehat{v}_{h}(s) - \overline{v}_{h}(s))_{H} 
    + (\overline{z}_{h}(s) - \overline{z}_{\tau}(s), 
                                          \overline{v}_{h}(s) - \overline{v}_{\tau}(s))_{H} 
\\ \notag 
&\,\quad + (\overline{z}_{h}(s) - \overline{z}_{\tau}(s), 
                                          \overline{v}_{\tau}(s) - \widehat{v}_{\tau}(s))_{H}.  
\end{align}
Here we derive from the second equation in \ref{Ph} and \eqref{rem7} that 
\begin{align}\label{s2}
&(\overline{z}_{h}(s) - \overline{z}_{\tau}(s), 
                                          \overline{v}_{h}(s) - \overline{v}_{\tau}(s))_{H}  
\\ \notag 
&= -\|\overline{v}_{h}(s) - \overline{v}_{\tau}(s)\|_{H}^2 
\\ \notag 
   &\,\quad+ \bigl(-a(\cdot)(\underline{\varphi}_{h}(s) - \underline{\varphi}_{\tau}(s)) 
                  + J\ast(\underline{\varphi}_{h}(s) - \underline{\varphi}_{\tau}(s)),  
                                                            \overline{v}_{h}(s) - \overline{v}_{\tau}(s)
                                                                                                         \bigr)_{H} 
\\ \notag 
    &\,\quad- \bigl(\beta(\overline{\varphi}_{h}(s)) - \beta(\overline{\varphi}_{\tau}(s)), 
                                                 \overline{v}_{h}(s) - \overline{v}_{\tau}(s) \bigr)_{H}
\\ \notag 
    &\,\quad- \bigl(\pi(\overline{\varphi}_{h}(s)) - \pi(\overline{\varphi}_{\tau}(s)), 
                                                 \overline{v}_{h}(s) - \overline{v}_{\tau}(s) \bigr)_{H} 
\\ \notag 
&\,\quad + \bigl(-h(\widehat{\theta}_{h})_{t}(s) + \tau(\widehat{\theta}_{\tau})_{t}(s), 
                                                \overline{v}_{h}(s) - \overline{v}_{\tau}(s)  \bigr)_{H} 
+ (\overline{\theta}_{h}(s) - \overline{\theta}_{\tau}(s), 
                                          \overline{v}_{h}(s) - \overline{v}_{\tau}(s))_{H}. 
\end{align}
The property \eqref{rem8} means that 
\begin{align}\label{s3}
\|\underline{\varphi}_{h}(s) - \underline{\varphi}_{\tau}(s)\|_{H}^2 
&= \|- h(\widehat{\varphi}_{h})_{t}(s) + \tau(\widehat{\varphi}_{\tau})_{t}(s) 
                                + \overline{\varphi}_{h}(s) - \overline{\varphi}_{\tau}(s) \|_{H}^2  
\\ \notag 
&\leq 3h^2 \|(\widehat{\varphi}_{h})_{t}(s) \|_{H}^2 
         + 3\tau^2 \|(\widehat{\varphi}_{\tau})_{t}(s) \|_{H}^2 
         + 3\|\overline{\varphi}_{h}(s) - \overline{\varphi}_{\tau}(s)\|_{H}^2. 
\end{align}
We can obtain that 
\begin{align}\label{s4}
\|\overline{\varphi}_{h}(s) - \overline{\varphi}_{\tau}(s)\|_{H}^2 
&= \|\overline{\varphi}_{h}(s) - \widehat{\varphi}_{h}(s) 
         + \widehat{\varphi}_{h}(s) - \widehat{\varphi}_{\tau}(s)   
              + \widehat{\varphi}_{\tau}(s) - \overline{\varphi}_{\tau}(s) \|_{H}^2 
\\ \notag 
&\leq 3\|\overline{\varphi}_{h}(s) - \widehat{\varphi}_{h}(s)\|_{H}^2 
          + 3\|\widehat{\varphi}_{h}(s) - \widehat{\varphi}_{\tau}(s)\|_{H}^2 
          + 3\|\widehat{\varphi}_{\tau}(s) - \overline{\varphi}_{\tau}(s) \|_{H}^2    
\end{align}
and 
\begin{align}\label{s5}
\frac{1}{2}\frac{d}{ds}\|\widehat{\varphi}_{h}(s) - \widehat{\varphi}_{\tau}(s) \|_{H}^2 
= (\overline{v}_{h}(s) - \overline{v}_{\tau}(s), 
                                     \widehat{\varphi}_{h}(s) - \widehat{\varphi}_{\tau}(s) )_{H}.  
\end{align}
It follows from the first equation in \ref{Ph} that 
\begin{align}\label{s6}
&\frac{1}{2}\frac{d}{ds}\|\widehat{\theta}_{h}(s) - \widehat{\theta}_{\tau}(s)\|_{H}^2 
\\ \notag 
&= -\bigl(-\Delta (\overline{\theta}_{h}(s) - \overline{\theta}_{\tau}(s)), 
                                    \widehat{\theta}_{h}(s) - \widehat{\theta}_{\tau}(s) \bigr)_{H}
- (\overline{v}_{h}(s) - \overline{v}_{\tau}(s), 
                     \widehat{\theta}_{h}(s) - \widehat{\theta}_{\tau}(s))_{H} 
\\ \notag 
&\,\quad+ (\overline{f}_{h}(s) - \overline{f}_{\tau}(s), 
                     \widehat{\theta}_{h}(s) - \widehat{\theta}_{\tau}(s))_{H}   
\\ \notag 
&= -\|\nabla(\overline{\theta}_{h}(s) - \overline{\theta}_{\tau}(s))\|_{H}^2 
     - \bigl(-\Delta (\overline{\theta}_{h}(s) - \overline{\theta}_{\tau}(s)), 
                                    \widehat{\theta}_{h}(s) - \overline{\theta}_{h}(s) \bigr)_{H}
\\ \notag      
     &\,\quad-\bigl(-\Delta (\overline{\theta}_{h}(s) - \overline{\theta}_{\tau}(s)), 
                                    \overline{\theta}_{\tau}(s) - \widehat{\theta}_{\tau}(s) \bigr)_{H}
     - (\overline{\theta}_{h}(s) - \overline{\theta}_{\tau}(s), 
                                  \overline{v}_{h}(s) - \overline{v}_{\tau}(s))_{H} 
\\ \notag    
  &\,\quad- (\overline{v}_{h}(s) - \overline{v}_{\tau}(s), 
               \widehat{\theta}_{h}(s) - \overline{\theta}_{h}(s))_{H} 
- (\overline{v}_{h}(s) - \overline{v}_{\tau}(s), 
                \overline{\theta}_{\tau}(s) - \widehat{\theta}_{\tau}(s))_{H} 
\\ \notag 
  &\,\quad+ (\overline{f}_{h}(s) - \overline{f}_{\tau}(s), 
                     \widehat{\theta}_{h}(s) - \widehat{\theta}_{\tau}(s))_{H}.     
\end{align}
Therefore we have from \eqref{s1}-\eqref{s6}, 
the integration over $(0, t)$, where $t \in [0, T]$, 
the Schwarz inequality, the Young inequality, 
(A1), Lemma \ref{timedisces3}, 
the local Lipschitz continuity of $\beta$, (A3), \eqref{rem4}-\eqref{rem6}, 
Lemmas \ref{timedisces2}, \ref{timedisces6} and \ref{timedisces7} that 
there exists a constant $C_{1}>0$ such that 
\begin{align*}
&\|\widehat{v}_{h}(t) - \widehat{v}_{\tau}(t)\|_{H}^2 
+ \int_{0}^{t}\|\overline{v}_{h}(s) - \overline{v}_{\tau}(s)\|_{H}^2\,ds 
+ \|\widehat{\varphi}_{h}(t) - \widehat{\varphi}_{\tau}(t)\|_{H}^2 
+ \|\widehat{\theta}_{h}(t) - \widehat{\theta}_{\tau}(t)\|_{H}^2 
\\ \notag 
&\,\quad+ \int_{0}^{t} 
                  \|\nabla(\overline{\theta}_{h}(s) - \overline{\theta}_{\tau}(s))\|_{H}^2\,ds  
\\ \notag 
&\leq C_{1}(h + \tau) + C_{1}\|\overline{f}_{h} - \overline{f}_{\tau}\|_{L^{2}(0, T; H)}^2   
        + C_{1}\int_{0}^{t}\|\widehat{\varphi}_{h}(s) - \widehat{\varphi}_{\tau}(s)\|_{H}^2\,ds 
\\ \notag  
   &\,\quad
          + C_{1}\int_{0}^{t}\|\widehat{\theta}_{h}(s) - \widehat{\theta}_{\tau}(s)\|_{H}^2\,ds 
\end{align*}
for all $h, \tau \in (0, h_{4})$.  
Then we can prove Lemma \ref{Cauchy's criterion} by the Gronwall lemma. 
\end{proof}

\vspace{10pt}

\section{Existence and uniqueness for \eqref{P} and an error estimate}\label{Sec4}
In this section we verify Theorems \ref{maintheorem2} and \ref{maintheorem3}. 
\begin{prth1.2}
Since $\overline{f}_{h}$ converges to $f$ strongly in $L^{2}(0, T; H)$ 
as $h \searrow 0$ (see \cite[Section 5]{CK1}), 
we see from Lemmas \ref{timedisces1}-\ref{Cauchy's criterion}, 
\eqref{rem4}-\eqref{rem8} that 
there exist some functions 
\begin{align*}
    &\theta \in H^1(0, T; H) \cap L^{\infty}(0, T; V) \cap L^{2}(0, T; W),  
    \\
    &\varphi \in W^{2, \infty}(0, T; H) \cap W^{2, 2}(0, T; L^{\infty}(\Omega)) 
                                                        \cap W^{1, \infty}(0, T; L^{\infty}(\Omega)) 
    \end{align*}
such that 
\begin{align}
\label{weak1} 
&\widehat{\theta}_{h} \to \theta  
\quad \mbox{weakly$^{*}$ in}\ H^1(0, T; H) \cap L^{\infty}(0, T; V),   
\\[1mm]
\notag 
&\overline{\theta}_{h} \to \theta     
\quad \mbox{weakly$^{*}$ in}\ L^{\infty}(0, T; V),  
\\[1mm] 
\label{weak2}
&\overline{\theta}_{h} \to \theta    
\quad \mbox{weakly in}\ L^{2}(0, T; W),  
\\[1mm]
\label{weak3}
&\underline{\theta}_{h} \to \theta 
\quad \mbox{weakly$^{*}$ in}\ L^{\infty}(0, T; V) \cap L^{2}(0, T; L^{\infty}(\Omega)),  
\\[1mm] 
\label{strong1}
&\widehat{\theta}_{h} \to \theta     
\quad \mbox{strongly in}\ C([0, T]; H),  
\\[1mm] 
\notag 
&\overline{\theta}_{h} \to \theta     
\quad \mbox{strongly in}\ L^{2}(0, T; V),   
\\[1mm]
\label{weak4}
&\widehat{\varphi}_{h} \to \varphi     
\quad \mbox{weakly$^{*}$ in}\ W^{1, \infty}(0, T; L^{\infty}(\Omega)),  
\\[1mm] 
\notag 
&\overline{\varphi}_{h} \to \varphi     
\quad \mbox{weakly$^{*}$ in}\ L^{\infty}(\Omega\times(0, T)),  
\\[1mm] 
\notag 
&\underline{\varphi}_{h} \to \varphi     
\quad \mbox{weakly$^{*}$ in}\ L^{\infty}(\Omega\times(0, T)),  
\\[1mm] 
\label{strong2}
&\widehat{\varphi}_{h} \to \varphi     
\quad \mbox{strongly in}\ C([0, T]; H),  
\\[1mm]
\notag 
&\widehat{v}_{h} \to \varphi_{t}  
\quad \mbox{weakly$^{*}$ in}\ W^{1, \infty}(0, T; H) \cap 
                        W^{1, 2}(0, T; L^{\infty}(\Omega)) \cap L^{\infty}(\Omega\times(0, T)),   
\\[1mm]
\label{weak5}
&\overline{v}_{h} \to \varphi_{t}  
\quad \mbox{weakly$^{*}$ in}\ L^{\infty}(\Omega\times(0, T)),  
\\[1mm]
\label{strong3}
&\widehat{v}_{h} \to \varphi_{t}     
\quad \mbox{strongly in}\ C([0, T]; H),  
\\[1mm]
\notag 
&\overline{v}_{h} \to \varphi_{t}     
\quad \mbox{strongly in}\ L^{2}(0, T; H),  
\\[1mm]
\label{weak6}
&\overline{z}_{h} \to \varphi_{tt}     
\quad \mbox{weakly$^{*}$ in}\ L^{\infty}(0, T; H) \cap L^{2}(0, T; L^{\infty}(\Omega))  
\end{align}
as $h = h_{j} \searrow 0$. 
Here we recall \eqref{rem5} and Lemma \ref{timedisces3} to 
derive from \eqref{strong2} that  
\begin{align*}
\|\overline{\varphi}_h - \varphi\|_{L^{\infty}(0, T; H)} 
&\leq \|\overline{\varphi}_h - \widehat{\varphi}_h\|_{L^{\infty}(0, T; H)} 
       + \|\widehat{\varphi}_h - \varphi\|_{L^{\infty}(0, T; H)} 
\\
&\leq |\Omega|^{1/2}
         \|\overline{\varphi}_h - \widehat{\varphi}_h\|_{L^{\infty}(0, T; L^{\infty}(\Omega))} 
        + \|\widehat{\varphi}_h - \varphi\|_{L^{\infty}(0, T; H)} 
\\    
&= \frac{|\Omega|^{1/2}}{\sqrt{3}}
                                     h\|\overline{v}_{h}\|_{L^{\infty}(0, T; L^{\infty}(\Omega))} 
     + \|\widehat{\varphi}_h - \varphi\|_{L^{\infty}(0, T; H)} 
\to 0  
\end{align*}
as $h = h_{j} \searrow 0$, and hence it holds that 
\begin{align}\label{strongoverlinevarphi}
\overline{\varphi}_h \to \varphi \quad\mbox{strongly in}\ L^{\infty}(0, T; H)
\end{align}
as $h = h_{j} \searrow 0$. 
Thus we infer from \eqref{rem8}, \eqref{strongoverlinevarphi} 
and Lemma \ref{timedisces7} that 
\begin{align}\label{strongunderlinevarphi}
\underline{\varphi}_h \to \varphi \quad\mbox{strongly in}\ L^{\infty}(0, T; H)
\end{align}
as $h = h_{j} \searrow 0$. 
Therefore, owing to \eqref{weak1}-\eqref{strongunderlinevarphi}, 
(A1), Lemma \ref{timedisces3}, the local Lipschitz continuity of $\beta$, 
and (A3), we can establish existence of solutions to \eqref{P}. 
Moreover, we can confirm uniqueness of solutions to \eqref{P} 
in a similar way to the proof of Lemma \ref{Cauchy's criterion}.  
\qed
\end{prth1.2}
\begin{prth1.3}
Since the inclusion  $f \in L^{2}(0, T; H) \cap W^{1,1}(0, T; H)$ 
implies that there exists a constant $C_{1}>0$ such that 
\[
\|\overline{f}_{h} - f\|_{L^{2}(0, T; H)} \leq C_{1} h^{1/2}
\]
for all $h >0$ (see \cite[Section 5]{CK1}), 
we can prove Theorem \ref{maintheorem3} by 
Lemma \ref{Cauchy's criterion}. 
\qed
\end{prth1.3}

\section*{Acknowledgments}
The author is supported by JSPS Research Fellowships 
for Young Scientists (No.\ 18J21006).   
%
%
%

\end{document}